\documentclass[12pt,reqno]{article}
\usepackage[usenames]{color}
\usepackage{amssymb}
\usepackage{amsmath}
\usepackage{amsthm}
\usepackage{amsfonts}
\usepackage{amscd}
\usepackage{graphicx}
\usepackage[usenames,dvipsnames]{xcolor}
\usepackage{color}
\usepackage{graphics}
\usepackage{latexsym}
\usepackage{epsf}

\usepackage[colorlinks=true,
linkcolor=webgreen,
filecolor=webbrown,
citecolor=webgreen]{hyperref}
\usepackage{hyphenat}
\definecolor{webgreen}{rgb}{0,.5,0}
\definecolor{webbrown}{rgb}{.6,0,0}

\usepackage{fullpage}
\usepackage{float}

\usepackage{graphics}
\usepackage{latexsym}
\usepackage{epsf}

\usepackage{tikz}
\usepackage{mathrsfs}

\begin{document}
	\theoremstyle{plain}
	\newtheorem{theorem}{Theorem}
	\newtheorem*{theorem*}{Theorem}
	
	\newtheorem{corollary}[theorem]{Corollary}
	\newtheorem{lemma}[theorem]{Lemma}
	\newtheorem{proposition}[theorem]{Proposition}
	\theoremstyle{definition}
	\newtheorem{definition}[theorem]{Definition}
	\newtheorem{notation}[theorem]{Notation}
	\newtheorem{example}[theorem]{Example}
	\newtheorem{conjecture}[theorem]{Conjecture}
	\theoremstyle{remark}
	\newtheorem{remark}[theorem]{Remark}
	
		\begin{center}
		\vskip 1cm{\Large \bf 
	A new explicit expansion approach to Mersenne primes}
		\vskip 1cm
	
		Moustafa Ibrahim\\
		Department of Mathematics\\
		College of Science, University of Bahrain, 
		 Kingdom of Bahrain\\
		\href{mailto:mimohamed@uob.edu.bh}{\tt mimohamed@uob.edu.bh}
	\end{center}
	
	\vskip .2 in
	
	\begin{abstract}
This paper first proves what the author called the Eight Levels Theorem and then highlights a new explicit expansion approach to Lucas-Lehmer primality test for Mersenne primes and gives a new criterion for Mersenne compositeness. Also, we prove four new combinatorial identities.
	\end{abstract}
		
		\section{Notations} For a natural number $n$, we define $\delta(n)=n\pmod{2}$. For an arbitrary real number $x$,  $\lfloor{\frac{x}{2}}\rfloor$ is the highest integer less than or equal $\frac{x}{2}$. We also need the following notation
\[ 	\prod\limits_{\lambda = 0}^{-1} n^2 - (4 \lambda)^2 := 1\]

\section{Introduction}
	 
	 Primes of special form have been of perennial interest, \cite{Guy}. Among these, the primes of the form 
	 	\[ 2^p -1\]
	which are called Mersenne prime. It is outstanding in their simplicity.

	\begin{itemize} 
		
		\item Clearly, if $2^p -1$ is prime then $p$ is prime. 
		The number  $2^p -1$ is Mersenne composite if $p$ is prime but $2^p -1$ is not prime. 
	\item	 Mathematics is kept alive by the appearance of challenging unsolved problems. The current paper gives new expansions to the following two major open questions in number theory (see \cite{Elina},  \cite{Mullen}, \cite{Jean}, \cite{Guy},  \cite{Washington}):
Are there infinitely many Mersenne primes? Are there infinitely many Mersenne composite?
		\item 
	 Mersenne primes have a close connection to perfect numbers, which are numbers that are equal to the sum of their proper divisors. It is known that Euclid and Euler proved that 
	 a number $N$ is even perfect number if and only if $N=2^{p-1}(2^p-1)$ for some prime $p$, and $2^p -1$ is prime. Euclid proved only that this statement was sufficient. Euler, 2000 years later, proved that all even perfect numbers are of the form $2^{p-1} (2^p -1)$ where $2^p-1$ is a Mersenne prime $M_p$ (see \cite{Jean, Penguin}). Thus the theorems of Euclid and Euler characterize all even perfect numbers, reducing their existence to that of Mersenne primes. 
	 	
	 	\item The odd perfect numbers are quite a different story (see \cite{Penguin, 25, Dickson1919}). It is unknown whether there is any odd perfect number. Recently, \cite{25} showed that odd perfect numbers, if exist, are greater than $10^{1500}$.   
	 	
	 	\item It is well-known the following two theorems  (see \cite{Jean}, \cite{Elina}, \cite{Washington}):
	 		
	 	\begin{theorem}{(Lucas-Lehmer)}
	 		\label{E0}\\
	 		$2^p -1$ is Mersenne prime \textbf{ if and only if}
	 		\begin{equation}
	 			\label{E00} 
	 			2n -1 \quad | \quad (1+\sqrt{3})^n + (1-\sqrt{3})^n
	 		\end{equation}
	 		where $n:=2^{p-1}$.	
	 	\end{theorem}

	 \begin{theorem}{(Euclid-Euler-Lucas-Lehmer)}
	 	\label{E}\\
	 	A number $N$ is even perfect number \textbf{ if and only if} $N=2^{p-1}(2^p-1)$ for some prime $p$, and 
	 	\begin{equation}
	 		\label{E1} 
	 		2n -1 \quad | \quad (1+\sqrt{3})^n + (1-\sqrt{3})^n
	 	\end{equation}
	 	where $n:=2^{p-1}$.	
	 \end{theorem}

	 \end{itemize}
 
 \subsection{The purpose of the paper} The aim of this paper is to study some arithmetical properties of the coefficients of the expansion 
 \begin{equation}
 	\label{E2} 
 x^n+y^n =
 \sum_{k=0}^{\lfloor{\frac{n}{2}}\rfloor} \Psi_k(n)
 (xy)^{\lfloor{\frac{n}{2}}\rfloor -k} (x^2+y^{2})^{k}
 \end{equation}

for $n \equiv 0, \: 2, \: 4, \: 6\pmod{8}$. Then we study the expansion 

 \begin{equation}
	\label{E3} 
	\frac{x^n+y^n}{x+y} =
	\sum_{k=0}^{\lfloor{\frac{n}{2}}\rfloor} \Psi_k(n)
	(xy)^{\lfloor{\frac{n}{2}}\rfloor -k} (x^2+y^{2})^{k}
\end{equation}
for $n \equiv 1, \: 3, \: 5, \: 7\pmod{8}$. Then we show that the numbers of the form 
  \begin{equation}
 	\label{E33} 
\prod\limits_{\lambda }^{} n^2 - (4 \lambda)^2
  \end{equation}
 arise up naturally in the coefficients of the expansions \eqref{E2}, \eqref{E3}, and enjoy some arithmetical properties.  We then be able to find three new expansions (three different versions) for Theorem \eqref{E0}, and a new expansion for Theorem \eqref{E}, which should give a better theoretical understanding for two major questions about integers; whether there are infinitely many Mersenne primes or Mersenne composites. We also prove four new combinatorial identities about the nature of integers.
 
\section{Summary for the main results of the paper}
	This paper proves all of the following new theorems for the primality of Mersenne numbers, even perfect numbers, proves a criterion for compositeness of Mersenne numbers, and gives four combinatorial identities related to the nature of numbers:

\begin{theorem}{(Lucas-Lehmer-Moustafa)}{(Version 1)}
	\label{Theorem of result-33T}\\
	Given prime $p \geq 5$. $2^p-1$ is prime  \textbf{ if and only if} 	
	\begin{equation}
		\label{Equation of result-33T} 
			2n-1 \quad  \vert \quad \sum_{\substack{k=0,\\ k \:even}}^{\lfloor{\frac{n}{2}}\rfloor}  \phi_k (n)
	\end{equation}
	where $n:=2^{p-1}$,  $\phi_k (n)$ are defined by the double index recurrence relation \[\phi_k (m)= 4 \: \phi_{k-1}(m-2) - \phi_k (m-4) \]
	and the initial boundary values satisfy 
	\begin{equation}
		\label{initial values - 1BBT} 
		\begin{aligned}	
			\phi_0(m) = 		
			\begin{cases}
				+2  &  m \equiv 0  \pmod{8}  \\   
				\: \: 0  &  m \equiv \pm 2  \pmod{8}  \\  
				-2  &  m \equiv 4  \pmod{8}  \\   
			\end{cases} 
			\qquad , \qquad 
			\phi_1(m) = 		
			\begin{cases}
				+ 2 \: m &  m \equiv 2  \pmod{8}  \\   
				\quad 0 &  m \equiv 0, \:4  \pmod{8}  \\   
				-2 \:  m  &  m \equiv 6  \pmod{8}  \\    
			\end{cases}, 
		\end{aligned} \\
	\end{equation}
	
\end{theorem}
	\begin{theorem}{(Lucas-Lehmer-Moustafa)}{(Version 2) }
	\label{Theorem of result-333T}\\
	Given prime $p \geq 5$,  $n:=2^{p-1}$. The number $2^p-1$ is prime  \textbf{ if and only if} 	
	\begin{equation}
		\label{Equation of result-333T} 
		2\:n -1  \quad  \vert \quad \sum_{\substack{k=0,\\ k \:even}}^{\lfloor{\frac{n}{2}}\rfloor}   \quad  (-1)^{\lfloor{\frac{k}{2}}\rfloor}	\: \frac {\prod\limits_{\lambda = 0}^{\lfloor{\frac{k}{2}}\rfloor -1} n^2 - (4 \lambda)^2}	{ k!}. 
	\end{equation}
\end{theorem}
	\begin{theorem}{(Lucas-Lehmer-Moustafa)}{(Version 3)}
		\label{Theorem of result-303}\\
		Given prime $p \geq 5$. The number $2^p-1$ is prime \textbf{ if and only if} 	
		\begin{equation}
			\label{Equation of result-303} 
			2 \: n - 1  \quad  \vert \quad \sum_{\substack{k=0,\\ k \:even}}^{\lfloor{\frac{n}{2}}\rfloor}  \phi_k (n)
		\end{equation}
		where $n:=2^{p-1}$,  $\phi_k (n)$ are generated by the double index recurrence relation
		\[   \frac{\phi_{k}(n)}{\phi_{k-2}(n)} = - \: \frac{\: \:	n^2  \: - \: (2k-4)^2 \:}{\: k \: (k-1)},
		\]
		and we can choose either of the following initial values to generate $\phi_k(n)$ from the starting term $\phi_0(n)$ or the last term $\phi_{\lfloor{\frac{n}{2}}\rfloor }(n)$ :
		\begin{equation}
			\label{S22} 
				\phi_0(n) = +2 		
				\qquad , \qquad 
				\phi_{\lfloor{\frac{n}{2}}\rfloor }(n) = 4^{\lfloor{\frac{n}{2}}\rfloor}.
		\end{equation}

		\begin{theorem}{(Euclid-Euler-Lucas-Lehmer-Moustafa)}
			\label{Theorem of result-2}\\
			A number $N$ is even perfect number \textbf{ if and only if} $N=2^{p-1}(2^p-1)$ for some prime $p$, and 
			\begin{equation}
				\label{Equation of result-2} 
				2 \:n -1  \quad  \vert \quad \sum_{\substack{k=0,\\ k \:even}}^{\lfloor{\frac{n}{2}}\rfloor} (-1)^{\lfloor{\frac{k}{2}}\rfloor}
				\frac {\prod\limits_{\lambda = 0}^{\lfloor{\frac{k}{2}}\rfloor -1} n^2 - (4 \lambda)^2}	{k!}
			\end{equation}
			where $n:=2^{p-1}$.	
		\end{theorem}
	
		\begin{theorem}{(Criteria for compositeness of Mersenne numbers ) }
			\label{Criteria}
			Given prime $p \geq 5$. The number  $2n-1=2^p -1$ is Mersenne composite number if
			\begin{equation}
				\label{Equation of result-BB} 
					2 \: n \: - \: 1 \quad  \nmid \quad \sum_{\substack{k=0,\\ k \:even}}^{\lfloor{\frac{n}{2}}\rfloor}   \quad  	\: \frac {\prod\limits_{\lambda = 0}^{\lfloor{\frac{k}{2}}\rfloor -1}  [(4 \lambda)^2 \: - \: 4^{-1}] \quad}	{ k!}. 
			\end{equation}
		\end{theorem} 	 
		
	\begin{theorem}{(Four new combinatorial identities)}
		\label{E4}
		For any natural number $n$, the following combinatorial identities are correct
		\begin{equation*}
			\begin{array}{l  l l}	
				\underline{n \equiv 0 \pmod{4}}  &  \quad & 4^{\lfloor{\frac{n}{2}}\rfloor} \: (\lfloor{\frac{n}{2}}\rfloor)! = 	2 \:
				\:\prod\limits_{\lambda = 0}^{\lfloor{\frac{n-4}{4}}\rfloor } n^2 \: - \: (4 \lambda)^2 	\\[1.5mm]
				
				\underline{n \equiv 1 \pmod{4}}  &  \quad&  4^{\lfloor{\frac{n}{2}}\rfloor} \: (\lfloor{\frac{n}{2}}\rfloor)! = 	\:
				\:\prod\limits_{\lambda = 1}^{\lfloor{\frac{n-1}{4}}\rfloor } (n+1)^2 \: - \: (4 \lambda -2)^2 	  
				\\[1.5mm]	
				\underline{n \equiv 2 \pmod{4}}  &  \quad& 4^{\lfloor{\frac{n}{2}}\rfloor} \: (\lfloor{\frac{n}{2}}\rfloor)! = 2	\: \: n 
				\:\prod\limits_{\lambda = 1}^{\lfloor{\frac{n-2}{4}}\rfloor } n^2 \: - \: (4 \lambda -2)^2 	  		
				\\[1.5mm]
				\underline{n \equiv 3 \pmod{4}}  &  \quad& 4^{\lfloor{\frac{n}{2}}\rfloor} \: (\lfloor{\frac{n}{2}}\rfloor)! =  \: (n+1) 
				\:\prod\limits_{\lambda = 1}^{\lfloor{\frac{n-3}{4}}\rfloor } (n+1)^2 \: - \: (4 \lambda)^2 	  
				\\
			\end{array} \\
		\end{equation*}
	
	\end{theorem}

	\section{The eight levels theorem}	
		Now we prove what we call the Eight Levels Theorem for the expansion of the polynomial \[ \frac{x^n+y^n}{(x+y)^{\delta(n)}}\]	
		 in terms of the symmetric polynomials $xy$ and $x^2+y^2$. Then we investigate some properties for the coefficients of that expansions, and give some applications and prove my results in the summary.

	\end{theorem}
		\subsection{The statement of the eight levels theorem}
		\begin{theorem}{(The Eight Levels Theorem)}
		\label{Theorem of the 4 levels}  \\
		For any complex numbers $x,y$, any non negative integers $n,k$, the coefficients $\Psi_k(n)$ of the expansion  
		\begin{equation}
			\label{Equation of result-4} 
			\frac{x^n+y^n}{(x+y)^{\delta(n)}} =
			\sum_{k=0}^{\lfloor{\frac{n}{2}}\rfloor} \Psi_k(n)
			(xy)^{\lfloor{\frac{n}{2}}\rfloor -k} (x^2+y^{2})^{k}
		\end{equation}
		are integers and 
		
		\begin{equation}
			\label{starting} 
				\Psi_0(n) = 		
				\begin{cases}
					+2  &  n \equiv \:\: 0  \pmod{8}  \\   
					+	 1 &  n \equiv \pm 1  \pmod{8}  \\
					\: \: 0  &  n \equiv \pm 2  \pmod{8}  \\   
					-1&  n \equiv \pm 3  \pmod{8}  \\
					-2  &  n \equiv \pm 4  \pmod{8}  \\   
				\end{cases} 
		\end{equation}
			and, for each $ 1 \le k \le  \lfloor{\frac{n}{2}}\rfloor$, the coefficients satisfy the following statements 
	\begin{itemize}	
\item  For $n \equiv 0, \: 2, \: 4, \: 6\pmod{8}$:\\ 
\end{itemize}

	\underline{$n \equiv 0 \pmod{8}$} 
		\begin{equation*}
				 \Psi_k(n) =
				\begin{cases}
					0  & \mbox{for $k$ odd } \\   
					2 \: (-1)^{\lfloor{\frac{k}{2}}\rfloor}
					\: \frac {\prod\limits_{\lambda = 0}^{\lfloor{\frac{k}{2}}\rfloor -1} n^2 \: - \: (4 \lambda)^2}	{4^k \: k!} & \mbox{for $k$ even } \end{cases}  				
			\end{equation*}
		
		\underline{$n \equiv 2 \pmod{8} $} \\
	\begin{equation*}
				\Psi_k(n) =
				\begin{cases}
					0  & \mbox{for $k$ even } \\   
					2 \: (-1)^{{\lfloor{\frac{k}{2}}\rfloor}} \:
					n \:	\frac {\prod\limits_{\lambda = 1}^{\lfloor{\frac{k}{2}}\rfloor} n^2 \: - \: (4 \lambda -2)^2}	{4^k \: k!} & \mbox{for $k$ odd } \end{cases}   
			\end{equation*}
					
	\underline{$n \equiv 4 \pmod{8}$} \\	
	\begin{equation*}	
		\Psi_k(n) =
		\begin{cases}
			0  & \mbox{for $k$ odd } \\   
			2 \: (-1)^{\lfloor{\frac{k}{2}}\rfloor +1}  \:
			\frac {\prod\limits_{\lambda = 0}^{\lfloor{\frac{k}{2}}\rfloor -1} n^2 \: - \: (4 \lambda)^2}	{4^k \: k!} & \mbox{for $k$ even } \end{cases}   
	\end{equation*}	
		
		\underline{$n \equiv 6 \pmod{8}$}  \\
		\begin{equation*}
			\Psi_k(n) =
			\begin{cases}
				0  & \mbox{for $k$ even } \\   
				2 \: (-1)^{\lfloor{\frac{k}{2}}\rfloor +1}
				\:  n \: \frac {\prod\limits_{\lambda = 1}^{\lfloor{\frac{k}{2}}\rfloor } n^2 \: - \: (4 \lambda -2)^2}	{4^k \: k!} & \mbox{for $k$ odd } \end{cases}   
		\end{equation*}
	
		\begin{itemize}	
		\item  For $n \equiv 1, \: 3, \: 5, \: 7\pmod{8}$:\\ 
	\end{itemize}
	
	\underline{$n \equiv 1 \pmod{8}$} 	
	\begin{equation*}
		\Psi_k(n) = (-1)^{\lfloor{\frac{k}{2}}\rfloor} (n+1-2k)^{\delta(k)}  \:
		\frac {\prod\limits_{\lambda = 1}^{\lfloor{\frac{k}{2}}\rfloor} (n+1)^2 \: - \: (4 \lambda -2)^2}	{4^k \: k!}  
	\end{equation*}

\underline{$n \equiv 3 \pmod{8}$}  \\		
	\begin{equation*}
				 \Psi_k(n) = (-1)^{\lfloor{\frac{k}{2}}\rfloor + \delta(k-1)} (n+1) (n+1-2k)^{\delta(k-1)}  
				\frac {\prod\limits_{\lambda = 1}^{\lfloor{\frac{k-1}{2}}\rfloor} (n+1)^2 - (4 \lambda)^2}	{4^k \: k!}  
	\end{equation*}		

			\underline{$n \equiv 5 \pmod{8}$ }  \\
			\begin{equation*}
				\Psi_k(n) = (-1)^{\lfloor{\frac{k}{2}}\rfloor +1}  \: (n+1-2k)^{\delta(k)} \:
				\frac {\prod\limits_{\lambda = 1}^{\lfloor{\frac{k}{2}}\rfloor }(n+1)^2 \: - \: (4 \lambda -2)^2}	{4^k \: k!}  
			\end{equation*}
		
	\underline{$n \equiv 7 \pmod{8}$} \\
			\begin{equation*}
				\Psi_k(n) = (-1)^{\lfloor{\frac{k}{2}}\rfloor + \delta(k)} (n+1) (n+1-2k)^{\delta(k-1)}  \:
				\frac {\prod\limits_{\lambda = 1}^{\lfloor{\frac{k-1}{2}}\rfloor } (n+1)^2 - (4 \lambda)^2}	{4^k \: k!}  \\ 
			\end{equation*}
		
	\end{theorem}
\begin{proof}

Put $(x,y)=(1, \sqrt{-1})$ in \eqref{Equation of result-4}, we immediately get \eqref{starting}. To prove the statements of Theorem \eqref{Theorem of the 4 levels}, we need to prove the following lemmas. 
	\begin{lemma}
		\label{lemma one}  
		For each natural number $n$, $ 1 \le k \le  \lfloor{\frac{n}{2}}\rfloor$, the coefficients  $\Psi_k(n)$ of the expansion of \eqref{Equation of result-4} are integers, unique and  and satisfy
		\begin{equation}
			\label{Equation of result-B} 
			\Psi_k(n) = \Psi_{k-1}(n-2)   -    \Psi_k(n-4).       	
		\end{equation}
	\end{lemma}
	\begin{proof} From the Fundamental Theorem on Symmetric Polynomials, \cite{Perron},  \cite{1}, we have a sequence of integers $\Psi_k(n)$ satisfy the expansion of \eqref{Equation of result-4}. From the algebraic independence of $xy, x^2+y^2$,
		the coefficients  $\Psi_k(n)$ of the expansion of \eqref{Equation of result-4} are unique. Now, multiply \eqref{Equation of result-4} by $xy(x^2+y^2)$, and noting $ (x+y)^{\delta(n+2)} =  (x+y)^{\delta(n)} =    (x+y)^{\delta(n-2)} $ and $ \lfloor{\frac{n+2}{2}}\rfloor = \lfloor{\frac{n}{2}}\rfloor +1$ and $ \lfloor{\frac{n-2}{2}}\rfloor = \lfloor{\frac{n}{2}}\rfloor -1 $ we get \begin{equation}
			\label{Equation of result-4A} 
			\begin{array}{cc}
				& \sum_{k=1}^{\lfloor{\frac{n}{2}}\rfloor +1} \Psi_{k-1}(n) 	(xy)^{\lfloor{\frac{n}{2}}\rfloor +2 -k} (x^2+y^{2})^{k} = \\
				& \sum_{k=0}^{\lfloor{\frac{n}{2}}\rfloor +1} \Psi_k(n+2)
				(xy)^{\lfloor{\frac{n}{2}}\rfloor +2 -k} (x^2+y^{2})^{k} + 
				\sum_{k=0}^{\lfloor{\frac{n}{2}}\rfloor -1} \Psi_k(n-2)
				(xy)^{\lfloor{\frac{n}{2}}\rfloor +2 -k} (x^2+y^{2})^{k}
			\end{array}
		\end{equation}
		Again from the algebraic independence of $xy, x^2+y^2$, and from \eqref{Equation of result-4A}, we get the following identity for any natural number $n$ 
		
	\begin{equation}
		\label{eq}	
		\Psi_{k}(n+2) =	\Psi_{k-1}(n) - \Psi_{k}(n-2).		
	\end{equation}
		
		Replace $n$ by $n-2$ in \eqref{eq}, we get \eqref{Equation of result-B}.
				
	\end{proof}

	\begin{lemma}
		\label{lemma two}  
		For every even natural number $n$, $ 0 \le k \le  \lfloor{\frac{n}{2}}\rfloor$, the following statements are true for each case:
		\begin{equation}
			\begin{array}{l c ll}	
				\underline{n \equiv 0 \: \text {or} \: 4 \pmod{8}}  &  \quad & \Psi_k(n) = 0    \quad \text{for}  \: k \: \text{odd}, \\
				\underline{n \equiv 2 \: \text {or} \: 6 \pmod{8}}  &  \quad & \Psi_k(n) = 0    \quad \text{for}  \: k \: \text{even}. \\
			\end{array}      	
		\end{equation}
	\end{lemma}
	\begin{proof}
		Consider $n$ even natural number, and replace $\delta(n)=0$ in \eqref{Equation of result-4} to get 
		\begin{equation}
			\label{Equation of result-4B} 
			x^n+y^n =
			\sum_{k=0}^{\lfloor{\frac{n}{2}}\rfloor} \Psi_k(n)
			(xy)^{\lfloor{\frac{n}{2}}\rfloor -k} (x^2+y^{2})^{k}
		\end{equation}
		Then replace $x$ by $-x$ in \eqref{Equation of result-4B}, we get  
		\begin{equation}
			\label{Equation of result-4C} 
			x^n+y^n =
			\sum_{k=0}^{\lfloor{\frac{n}{2}}\rfloor} \Psi_k(n)
			(-xy)^{\lfloor{\frac{n}{2}}\rfloor -k} (x^2+y^{2})^{k}
		\end{equation}
		From the algebraic independence of $xy, x^2+y^2$, and from \eqref{Equation of result-4B}, \eqref{Equation of result-4C} we get the proof.
	\end{proof}
	From \eqref{Equation of result-4}, we get the following initial  information about  $\Psi_k(n)$.
	\begin{lemma}
		\label{Equation of result-4D} 
		\begin{equation*}
			\begin{array}{l l ll}
				\Psi_0(0) = +2, & \Psi_1(0) = 0,  & \Psi_2(0) =0, & \Psi_3(0) = 0,  \\  
				\Psi_0(2) = 0, & \Psi_1(2) = +1,  & \Psi_2(2) = 0, & \Psi_3(2) = 0,  \\  
				\Psi_0(4) = -2, & \Psi_1(4) = 0,  & \Psi_2(4) = +1, & \Psi_3(4) = 0,  \\  
				\Psi_0(6) =0, & \Psi_1(6) = -3,  & \Psi_2(6) = 0, & \Psi_3(6) = +1.  \\  
			\end{array} \\
		\end{equation*}
	\end{lemma}
	
	\begin{lemma}
		\label{lemma odd}  
		For any odd natural number $n$, the coefficients  $\Psi_k(n)$ of the expansion of \eqref{Equation of result-4} satisfy the following property
		\begin{equation}
			\label{equation of the odd} 
			\Psi_k(n) = \frac{2(k+1)}{(n+1)}\Psi_{k+1}(n+1) + \frac{\big(\lfloor{\frac{n+1}{2}}\rfloor -k\big)}{(n+1)}    \Psi_k(n+1)       	
		\end{equation}
	\end{lemma}
	\begin{proof} Consider $n$ odd, then $n+1$ even. Then from the expansion of \eqref{Equation of result-4} we get the following 
		\begin{equation}
			\label{expansion of n+1} 
			x^{n+1}+y^{n+1} =
			\sum_{k=0}^{\lfloor{\frac{n+1}{2}}\rfloor} \Psi_k(n+1)
			(xy)^{\lfloor{\frac{n+1}{2}}\rfloor -k} (x^2+y^{2})^{k}
		\end{equation}	
		
		Now acting  the differential operator $  ( \frac{{\partial} }{\partial x} +   \frac{{\partial} }{\partial y}) $ on \eqref{expansion of n+1} and noting that 
		\begin{equation*}
			\begin{aligned}
				( \frac{{\partial} }{\partial x} +   \frac{{\partial} }{\partial y}) ( x^{n+1}+y^{n+1}  )    &= (n+1)(x^{n}+y^{n} ),  \\ 
				( \frac{{\partial} }{\partial x} +   \frac{{\partial} }{\partial y}) xy    &= x+y,  \\ 
				( \frac{{\partial} }{\partial x} +   \frac{{\partial} }{\partial y}) ( x^{2}+y^{2}  )    &= 2(x +y ),  \\ 						
			\end{aligned}
		\end{equation*}
		and equating the coefficients, we get the proof. 
	\end{proof}
	
	\subsection{The proof of Theorem \eqref{Theorem of the 4 levels} for $n \equiv 0, \: 2, \: 4, \: 6\pmod{8}$}
	
	Now in this section we prove Theorem \eqref{Theorem of the 4 levels} for $n$ even. The values of  $\Psi_k(n)$  that comes from the formulas of Theorem \eqref{Theorem of the 4 levels} for $n = 0, ,2,4,8$ are identical with the correct values that come from Lemma \eqref{Equation of result-4D}. So, Theorem \eqref{Theorem of the 4 levels} is correct for the initial vales $n=0,2,4,6$. Now we assume the validity of Theorem \eqref{Theorem of the 4 levels} for each $ 0 \le m \: \textless \: n$, with $m \equiv 0, \: 2, \: 4, \: 6\pmod{8}$ and need to prove that Theorem \eqref{Theorem of the 4 levels} for $n$. Lemma \eqref{lemma two} proves the validity of Theorem \eqref{Theorem of the 4 levels} for $n \equiv 0 \: \text {or} \: 4 \pmod{8}$ if $k$ odd, and for $n \equiv 2 \: \text {or} \: 6 \pmod{8}$ if $k$ even. Therefore it remains to prove the validity of the following cases: 
	
	\begin{itemize}
		\item  $ n \equiv 0  \pmod{8} , k \text \: {even}$
		\item  $ n \equiv 2  \pmod{8} , k \text \: {odd}$
		\item  $ n \equiv 4  \pmod{8} , k \text \: {even}$
		\item  $ n \equiv 6  \pmod{8} , k \text \: {odd}$
	\end{itemize}
	We prove these cases one by one as following. \\   
	\underline{ Consider $ n \equiv 0  \pmod{8} \: \text{and} \: k \text \: {even}$} \\ 
	\begin{proof}
		In this case, $ n -2 \equiv 6  \pmod{8}$ and $ n -4 \equiv 4  \pmod{8} $ and from Lemma \eqref{lemma one}, we get
		\begin{equation*}
			\begin{array}{lll}
				\Psi_k(n) &= \Psi_{k-1}(n-2)   -    \Psi_k(n-4)  \\
				&= 2 \: (-1)^{\lfloor{\frac{k-1}{2}}\rfloor +1}
				\:  (n-2) \: \frac {\prod\limits_{\lambda = 1}^{\lfloor{\frac{k-1}{2}}\rfloor } (n-2)^2 - (4 \lambda -2)^2}	{4^{k-1} \: (k-1)!}  - 2 \: (-1)^{\lfloor{\frac{k}{2}}\rfloor +1}  \:
				\frac {\prod\limits_{\lambda = 0}^{\lfloor{\frac{k}{2}}\rfloor -1} (n-4)^2 - (4 \lambda)^2}	{4^k \: k!} \\
				&= 2 \: (-1)^{\lfloor{\frac{k}{2}}\rfloor }
				\:  (n-2) \: \frac {\prod\limits_{\lambda = 1}^{\lfloor{\frac{k}{2}}\rfloor -1 } (n - 4 \lambda) \:  \prod\limits_{\lambda = 1}^{\lfloor{\frac{k}{2}}\rfloor -1 } (n + 4 \lambda -4)     }	{4^{k-1} \: (k-1)!} 
				+ 2 \: (-1)^{\lfloor{\frac{k}{2}}\rfloor }
				\:  \frac {\prod\limits_{\lambda = 0}^{\lfloor{\frac{k}{2}}\rfloor -1 } (n - 4 \lambda -4) \:  \prod\limits_{\lambda = 0}^{\lfloor{\frac{k}{2}}\rfloor -1 } (n + 4 \lambda -4)     }	{4^{k-1} \: (k-1)!}  \\
				&= 2 \:\frac{(-1)^{\lfloor{\frac{k}{2}}\rfloor }}{4^k \: k!}  \prod\limits_{\lambda = 1}^{\lfloor{\frac{k}{2}}\rfloor -1 } (n - 4 \lambda) \:  \prod\limits_{\lambda = 0}^{\lfloor{\frac{k}{2}}\rfloor -2 } (n + 4 \lambda ) \: \big[  4k(n-2) + (n-4 \lfloor{\frac{k}{2}}\rfloor )(n-4)   \big]	\\	
				&= 2 \:\frac{(-1)^{\lfloor{\frac{k}{2}}\rfloor }}{4^k \: k!}  \prod\limits_{\lambda = 1}^{\lfloor{\frac{k}{2}}\rfloor -1 } (n - 4 \lambda) \:  \prod\limits_{\lambda = 0}^{\lfloor{\frac{k}{2}}\rfloor -2 } (n + 4 \lambda ) \quad n \: (  n+ 4 (\lfloor{\frac{k}{2}}\rfloor -1))\\	
				&= 2 \:\frac{(-1)^{\lfloor{\frac{k}{2}}\rfloor }}{4^k \: k!}  \prod\limits_{\lambda = 0}^{\lfloor{\frac{k}{2}}\rfloor -1 } (n - 4 \lambda) \:  \prod\limits_{\lambda = 0}^{\lfloor{\frac{k}{2}}\rfloor -1 } (n + 4 \lambda ) =  2 \: (-1)^{\lfloor{\frac{k}{2}}\rfloor}
				\: \frac {\prod\limits_{\lambda = 0}^{\lfloor{\frac{k}{2}}\rfloor -1} n^2 - (4 \lambda)^2}	{4^k \: k!} \\
			\end{array} \\
		\end{equation*}
	\end{proof}
	\underline{ Consider $ n \equiv 2  \pmod{8} \: \text{and} \: k \text \: {odd}$} \\ 
	\begin{proof}
		In this case, $ n -2 \equiv 0  \pmod{8}$ and $ n -4 \equiv 6  \pmod{8} $ and from Lemma \eqref{lemma one}, we get
		
		\begin{equation*}
			\begin{array}{lll}
				\Psi_k(n) &= \Psi_{k-1}(n-2)   -    \Psi_k(n-4)  \\
				&= 2 \: (-1)^{\lfloor{\frac{k-1}{2}}\rfloor }
				\: \frac {\prod\limits_{\lambda = 0}^{\lfloor{\frac{k-1}{2}}\rfloor -1 } (n-2)^2 - (4 \lambda)^2}	{4^{k-1} \: (k-1)!}  - 2 \: (-1)^{\lfloor{\frac{k}{2}}\rfloor +1}  \:
				(n-4)\frac {\prod\limits_{\lambda = 1}^{\lfloor{\frac{k-1}{2}}\rfloor } (n-4)^2 - (4 \lambda-2)^2}	{4^k \: k!} \\
				
				&= 2 \: (-1)^{\lfloor{\frac{k}{2}}\rfloor }
				\:  \frac {\prod\limits_{\lambda = 0}^{\lfloor{\frac{k}{2}}\rfloor -1 } (n - 4 \lambda -2) \:  \prod\limits_{\lambda = 0}^{\lfloor{\frac{k}{2}}\rfloor -1 } (n + 4 \lambda -2)     }	{4^{k-1} \: (k-1)!} 
				+ 2 \: (-1)^{\lfloor{\frac{k}{2}}\rfloor }
				\: (n-4) \frac {\prod\limits_{\lambda = 1}^{\lfloor{\frac{k}{2}}\rfloor } (n - 4 \lambda -2) \:  \prod\limits_{\lambda = 1}^{\lfloor{\frac{k}{2}}\rfloor  } (n + 4 \lambda -6)     }	{4^{k} \: k!}  \\
				&= 2 \:\frac{(-1)^{\lfloor{\frac{k}{2}}\rfloor }}{4^k \: k!} \prod\limits_{\lambda = 1}^{\lfloor{\frac{k}{2}}\rfloor -1 } (n - 4 \lambda -2) \:  \prod\limits_{\lambda = 0}^{\lfloor{\frac{k}{2}}\rfloor -1 } (n + 4 \lambda -2 ) \big[ 4k(n-2) +  (n-4)(n-4(\lfloor{\frac{k}{2}}\rfloor)   -2)   \big]	\\				
			\end{array} \\
		\end{equation*}
		As \[  4k(n-2) +  (n-4)(n-4(\lfloor{\frac{k}{2}}\rfloor)   -2 ) = n (n+2k -4),     \]
		we get the following 
		\begin{equation*}
			\begin{aligned}	
				\Psi_k(n) &= 2 \:\frac{(-1)^{\lfloor{\frac{k}{2}}\rfloor }}{4^k \: k!} \: n \: \prod\limits_{\lambda = 1}^{\lfloor{\frac{k}{2}}\rfloor } (n - 4 \lambda + 2) \:  \prod\limits_{\lambda = 1}^{\lfloor{\frac{k}{2}}\rfloor } (n + 4 \lambda -2) 	\\	
				&= 2 \: n \: (-1)^{\lfloor{\frac{k}{2}}\rfloor}
				\: \quad \frac {\prod\limits_{\lambda = 1}^{\lfloor{\frac{k}{2}}\rfloor } n^2 - (4 \lambda -2)^2}	{4^k \: k!} \\
			\end{aligned} \\
		\end{equation*}
	\end{proof}
	\underline{ Consider $ n \equiv 4  \pmod{8} \: \text{and} \: k \text \: {even}$} \\ 
	\begin{proof}
		In this case, $ n -2 \equiv 2  \pmod{8}$ and $ n - 4 \equiv 0  \pmod{8} $ and from Lemma \eqref{lemma one}, we get
		\begin{equation*}
			\begin{array}{lll}
				\Psi_k(n) &= \Psi_{k-1}(n-2)   -    \Psi_k(n-4)  \\
				&= 2 \: (-1)^{\lfloor{\frac{k-1}{2}}\rfloor }
				\:  (n-2) \: \frac {\prod\limits_{\lambda = 1}^{\lfloor{\frac{k-2}{2}}\rfloor } (n-2)^2 - (4 \lambda -2)^2}	{4^{k-1} \: (k-1)!}  - 2 \: (-1)^{\lfloor{\frac{k}{2}}\rfloor }  \:
				\frac {\prod\limits_{\lambda = 0}^{\lfloor{\frac{k}{2}}\rfloor -1} (n-4)^2 - (4 \lambda)^2}	{4^k \: k!} \\
				
				&= 2 \: (-1)^{\lfloor{\frac{k}{2}}\rfloor +1}
				\:  (n-2) \: \frac {\prod\limits_{\lambda = 1}^{\lfloor{\frac{k}{2}}\rfloor -1 } (n - 4 \lambda) \:  \prod\limits_{\lambda = 1}^{\lfloor{\frac{k}{2}}\rfloor -1 } (n + 4 \lambda -4)     }	{4^{k-1} \: (k-1)!} 
				+ 2 \: (-1)^{\lfloor{\frac{k}{2}}\rfloor +1 }
				\:  \frac {\prod\limits_{\lambda = 0}^{\lfloor{\frac{k}{2}}\rfloor -1 } (n - 4 \lambda -4) \:  \prod\limits_{\lambda = 0}^{\lfloor{\frac{k}{2}}\rfloor -1 } (n + 4 \lambda -4)     }	{4^k \: k!}  \\
			\end{array} \\
		\end{equation*}
		Hence 		
		\begin{equation*}
			\begin{aligned}	
				\Psi_k(n)	&= 2 \:\frac{(-1)^{\lfloor{\frac{k}{2}}\rfloor +1}}{4^k \: k!}  \prod\limits_{\lambda = 1}^{\lfloor{\frac{k}{2}}\rfloor -1 } (n - 4 \lambda) \:  \prod\limits_{\lambda = 0}^{\lfloor{\frac{k}{2}}\rfloor -2 } (n + 4 \lambda ) \quad n \: (  n+ 4 (\lfloor{\frac{k}{2}}\rfloor -1))\\	
				&= 2 \:\frac{(-1)^{\lfloor{\frac{k}{2}}\rfloor +1}}{4^k \: k!}  \prod\limits_{\lambda = 0}^{\lfloor{\frac{k}{2}}\rfloor -1 } (n - 4 \lambda) \:  \prod\limits_{\lambda = 0}^{\lfloor{\frac{k}{2}}\rfloor -1 } (n + 4 \lambda ) =  2 \: (-1)^{\lfloor{\frac{k}{2}}\rfloor +1}
				\: \frac {\prod\limits_{\lambda = 0}^{\lfloor{\frac{k}{2}}\rfloor -1} n^2 - (4 \lambda)^2}	{4^k \: k!} \\			
			\end{aligned} \\
		\end{equation*}
	\end{proof}
	\underline{ Consider $ n \equiv 6  \pmod{8} \: \text{and} \: k \text \: {odd}$} \\ 
	\begin{proof}
		In this case, $ n -2 \equiv 4  \pmod{8}$ and $ n -4 \equiv 2  \pmod{8} $ and from Lemma \eqref{lemma one}, we get
		\begin{equation*}
			\begin{array}{lll}
				\Psi_k(n) &= \Psi_{k-1}(n-2)   -    \Psi_k(n-4)  \\
				
				&= 2 \: (-1)^{\lfloor{\frac{k-1}{2}}\rfloor +1}
				\: \frac {\prod\limits_{\lambda = 0}^{\lfloor{\frac{k-1}{2}}\rfloor -1 } (n-2)^2 - (4 \lambda)^2}	{4^{k-1} \: (k-1)!}  - 2 \: (-1)^{\lfloor{\frac{k}{2}}\rfloor }  \:
				(n-4)\frac {\prod\limits_{\lambda = 1}^{\lfloor{\frac{k-1}{2}}\rfloor } (n-4)^2 - (4 \lambda-2)^2}	{4^k \: k!} \\
				&= 2 \: (-1)^{\lfloor{\frac{k}{2}}\rfloor +1}
				\:  \frac {\prod\limits_{\lambda = 0}^{\lfloor{\frac{k}{2}}\rfloor -1 } (n - 4 \lambda -2) \:  \prod\limits_{\lambda = 0}^{\lfloor{\frac{k}{2}}\rfloor -1 } (n + 4 \lambda -2)     }	{4^{k-1} \: (k-1)!} 
				+ 2 \: (-1)^{\lfloor{\frac{k}{2}}\rfloor +1}
				\: (n-4) \frac {\prod\limits_{\lambda = 1}^{\lfloor{\frac{k}{2}}\rfloor } (n - 4 \lambda -2) \:  \prod\limits_{\lambda = 1}^{\lfloor{\frac{k}{2}}\rfloor  } (n + 4 \lambda -6)     }	{4^{k} \: k!}  \\	
			\end{array} \\
		\end{equation*}	
		Hence 	
		\begin{equation*}
			\begin{aligned}	
				\Psi_k(n) &= 2 \:\frac{(-1)^{\lfloor{\frac{k}{2}}\rfloor +1}}{4^k \: k!} \: n \: \prod\limits_{\lambda = 1}^{\lfloor{\frac{k}{2}}\rfloor } (n - 4 \lambda + 2) \:  \prod\limits_{\lambda = 1}^{\lfloor{\frac{k}{2}}\rfloor } (n + 4 \lambda -2) 	\\	
				&= 2 \: n \: (-1)^{\lfloor{\frac{k}{2}}\rfloor +1}
				\: \quad \frac {\prod\limits_{\lambda = 1}^{\lfloor{\frac{k}{2}}\rfloor } n^2 - (4 \lambda -2)^2}	{4^k \: k!} \\
			\end{aligned} \\
		\end{equation*}
	\end{proof}
	
	\subsection{The proof of Theorem \eqref{Theorem of the 4 levels} for $n \equiv 1, \: 3, \: 5, \: 7\pmod{8}$}
	With the help of Lemma \eqref{lemma odd}, we rely on Theorem \eqref{Theorem of the 4 levels} for the even case that we already proved, together with Lemma \eqref{lemma two}, to prove Theorem \eqref{Theorem of the 4 levels} for the odd cases $n \equiv 1, \: 3, \: 5, \: 7\pmod{8},$ one by one for each parity for $k$.\\
	
	\underline{Consider $ n \equiv 1  \pmod{8} \: \text{and} \: k \text \: {odd}$} 
	\begin{proof}
		In this case, $ n +1 \equiv 2  \pmod{8}$ and from Lemma \eqref{lemma two}, we get $ \Psi_{k+1}(n+1)=0$. Hence, from Lemma \eqref{lemma odd}, and from Theorem \eqref{Theorem of the 4 levels}, we get the following relation 	
		\begin{equation*}
			\begin{aligned}	
				\Psi_k(n) =  \frac{\big(\lfloor{\frac{n+1}{2}}\rfloor -k\big)}{(n+1)}    \Psi_k(n+1)   
				=  (-1)^{\lfloor{\frac{k}{2}}\rfloor} (n+1-2k)  \:
				\frac {\prod\limits_{\lambda = 1}^{\lfloor{\frac{k}{2}}\rfloor} (n+1)^2 - (4 \lambda -2)^2}	{4^k \: k!} \\							    
			\end{aligned} \\
		\end{equation*}
	\end{proof}

	\underline{Consider $ n \equiv 1  \pmod{8} \: \text{and} \: k \text \: {even}$} 
	\begin{proof}
		In this case, $ n +1 \equiv 2  \pmod{8}$ and from Lemma \eqref{lemma two}, we get $ \Psi_{k}(n+1)=0$. From Lemma \eqref{lemma odd}, and from Theorem \eqref{Theorem of the 4 levels}, we get the following relation 	
		\begin{equation*}
			\begin{aligned}	
				\Psi_k(n) =  2 \frac{(k+1)}{(n+1)} \Psi_{k+1}(n+1)   
				=(-1)^{\lfloor{\frac{k}{2}}\rfloor}  \:
				\frac {\prod\limits_{\lambda = 1}^{\lfloor{\frac{k}{2}}\rfloor} (n+1)^2 - (4 \lambda -2)^2}	{4^k \: k!} \\							    
			\end{aligned} \\
		\end{equation*}
	\end{proof}
	
	\underline{Consider $ n \equiv 3 \pmod{8} \: \text{and} \: k \text \: {odd}$} 
	\begin{proof}
		In this case, from Lemma \eqref{lemma two}, we get $ \Psi_{k}(n+1)=0$. Hence, from Lemma \eqref{lemma odd}, and from Theorem \eqref{Theorem of the 4 levels}, we get the following relation 	
		
		\begin{equation*}
			\begin{aligned}	
				\Psi_k(n) =  2 \frac{(k+1)}{(n+1)} \Psi_{k+1}(n+1)   
				= (-1)^{\lfloor{\frac{k}{2}}\rfloor} (n+1) \:
				\frac {\prod\limits_{\lambda = 1}^{\lfloor{\frac{k}{2}}\rfloor } (n+1)^2 - (4 \lambda)^2}	{4^k \: k!} 
				\\							    
			\end{aligned} \\
		\end{equation*}
	\end{proof}
	
	\underline{Consider $ n \equiv 3  \pmod{8} \: \text{and} \: k \text \: {even}$} 
	\begin{proof}
		In this case, $ \Psi_{k+1}(n+1)=0$. From Lemma \eqref{lemma odd}, and from Theorem \eqref{Theorem of the 4 levels}, we get the following relation
		
		\begin{equation*}
			\begin{aligned}	
				\Psi_k(n) =  \frac{\big(\lfloor{\frac{n+1}{2}}\rfloor -k\big)}{(n+1)}    \Psi_k(n+1)   
				= (-1)^{\lfloor{\frac{k}{2}}\rfloor + 1} (n+1) (n+1-2k) \:
				\frac {\prod\limits_{\lambda = 1}^{\lfloor{\frac{k}{2}}\rfloor - 1} (n+1)^2 - (4 \lambda)^2}	{4^k \: k!}   \\							    
			\end{aligned} \\
		\end{equation*}
		
	\end{proof}

	\underline{Consider $ n \equiv 5  \pmod{8} \: \text{and} \: k \text \: {odd}$} 
	\begin{proof}
		In this case, from Lemma \eqref{lemma two}, we get $ \Psi_{k+1}(n+1)=0$. Hence, from Lemma \eqref{lemma odd}, and from Theorem \eqref{Theorem of the 4 levels}, we get the following relation 	
		\begin{equation*}
			\begin{aligned}	
				\Psi_k(n) =  \frac{\big(\lfloor{\frac{n+1}{2}}\rfloor -k\big)}{(n+1)}    \Psi_k(n+1)   
				=  (-1)^{\lfloor{\frac{k}{2}}\rfloor +1}  \: (n+1-2k) \:
				\frac {\prod\limits_{\lambda = 1}^{\lfloor{\frac{k}{2}}\rfloor }(n+1)^2 - (4 \lambda -2)^2}	{4^k \: k!}  \\							    
			\end{aligned} \\
		\end{equation*}
	\end{proof}

	\underline{ Consider $ n \equiv 5  \pmod{8} \: \text{and} \: k \text \: {even}$} 
	\begin{proof}
		In this case, from Lemma \eqref{lemma two}, we get $ \Psi_{k}(n+1)=0$. From Lemma \eqref{lemma odd}, and from Theorem \eqref{Theorem of the 4 levels}, we get the following relation 	
		\begin{equation*}
			\begin{aligned}	
				\Psi_k(n) =  2 \frac{(k+1)}{(n+1)} \Psi_{k+1}(n+1)   
				= (-1)^{\lfloor{\frac{k}{2}}\rfloor +1}  \: 
				\frac {\prod\limits_{\lambda = 1}^{\lfloor{\frac{k}{2}}\rfloor }(n+1)^2 - (4 \lambda -2)^2}	{4^k \: k!}  \\							    
			\end{aligned} \\
		\end{equation*}
	\end{proof}
	
	\underline{Consider $ n \equiv 7 \pmod{8} \: \text{and} \: k \text \: {odd}$} 
	\begin{proof}
		In this case, from Lemma \eqref{lemma two}, we get $ \Psi_{k}(n+1)=0$. Hence, from Lemma \eqref{lemma odd}, and from Theorem \eqref{Theorem of the 4 levels}, we get the following relation 	
		
		\begin{equation*}
			\begin{aligned}	
				\Psi_k(n) &=  2 \frac{(k+1)}{(n+1)} \Psi_{k+1}(n+1) \\  
				&= (-1)^{\lfloor{\frac{k}{2}}\rfloor + 1} (n+1)  \:
				\frac {\prod\limits_{\lambda = 1}^{\lfloor{\frac{k}{2}}\rfloor} (n+1)^2 - (4 \lambda)^2}	{4^k \: k!} 
				\\							    
			\end{aligned} \\
		\end{equation*}
	\end{proof}
	
	\underline{Consider $ n \equiv 7  \pmod{8} \: \text{and} \: k \text \: {even}$} 
	\begin{proof}
		In this case, $ \Psi_{k+1}(n+1)=0$. From Lemma \eqref{lemma odd}, and from Theorem \eqref{Theorem of the 4 levels}, we get the following relation
		
		\begin{equation*}
			\begin{aligned}	
				\Psi_k(n) &=  \frac{\big(\lfloor{\frac{n+1}{2}}\rfloor -k\big)}{(n+1)}    \Psi_k(n+1)  \\ 
				&= (-1)^{\lfloor{\frac{k}{2}}\rfloor } (n+1) (n+1-2k) \:
				\frac {\prod\limits_{\lambda = 1}^{\lfloor{\frac{k}{2}}\rfloor -1} (n+1)^2 - (4 \lambda)^2}	{4^k \: k!}  \\							    
			\end{aligned} \\
		\end{equation*}
		
	\end{proof}

This completes the proof of Theorem \eqref{Theorem of the 4 levels}.
	\end{proof}
	
	\section{General characteristics for $\Psi-$ sequence
	}
	
	\subsection{Examples for $\Psi-$Sequence. }
	From Theorem \eqref{Theorem of the 4 levels}, we list some examples that show the splendor of the natural factorization of $\Psi_k(n)$ for $k=0,1,2,3,4,5,6,7$: 
	\begin{equation}
		\label{initial values - 1} 
		\begin{aligned}	
			\Psi_0(n) = 		
			\begin{cases}
				+2  &  n \equiv 0  \pmod{8}  \\   
				+	 1 &  n \equiv 1  \pmod{8}  \\
				\: \: 0  &  n \equiv 2  \pmod{8}  \\   
				-1&  n \equiv 3  \pmod{8}  \\
				-2  &  n \equiv 4  \pmod{8}  \\   
				-1 &  n \equiv 5  \pmod{8}  \\
				\: \: 	0 &  n \equiv 6  \pmod{8}  \\   
				+	1  &  n \equiv 7 \pmod{8}  
			\end{cases} 
			\qquad , \qquad 
			\Psi_1(n) = 		
			\begin{cases}
				\quad 	0  &  n \equiv 0  \pmod{8}  \\   
				+	\frac{(n-1)}{4} &  n \equiv 1  \pmod{8}  \\[1.5mm]
				+ 2 \: \frac{n}{4} &  n \equiv 2  \pmod{8}  \\[1.5mm]  
				+\frac{(n+1)}{4} &  n \equiv 3  \pmod{8}  \\
				\quad 0 &  n \equiv 4  \pmod{8}  \\   
				- \frac{(n-1)}{4} &  n \equiv 5  \pmod{8}  \\[1.5mm]
				-2 \:  \frac{n}{4}  &  n \equiv 6  \pmod{8}  \\[1.5mm]  
				- \frac{(n+1)}{4}  &  n \equiv 7 \pmod{8}  
			\end{cases}, 
		\end{aligned} \\
	\end{equation}
	
	\begin{equation}
		\label{initial values - 2 } 
		\begin{aligned}	
			\Psi_2(n)  = 		
			\begin{cases}
				- 2 \: \frac{n^2}{4^2  \: 2!}  &  n \equiv 0  \pmod{8}  \\[1.5mm]   
				-	\frac{(n-1)(n+3)}{4^2  \: 2!} &  n \equiv 1  \pmod{8}  \\[1.5mm]
				\quad  0  &  n \equiv 2  \pmod{8}  \\   
				+\frac{(n-3)(n+1)}{4^2  \: 2!} &  n \equiv 3  \pmod{8}  \\[1.5mm]
				+ 2  \: \frac{n^2}{4^2  \: 2!}  &  n \equiv 4  \pmod{8}  \\[1.5mm]   
				+ \frac{(n-1)(n+3)}{4^2 \: 2!} &  n \equiv 5  \pmod{8}  \\[1.5mm]
				\quad 0  &  n \equiv 6  \pmod{8}  \\   
				- \frac{(n-3)(n+1)}{4^2  \:  2!}  &  n \equiv 7 \pmod{8}  
			\end{cases} 
			, \:
			\Psi_3(n) = 		
			\begin{cases}
				\quad 	0  &  n \equiv 0  \pmod{8}  \\   
				-	\frac{(n-5)(n-1)(n+3)}{4^3 \: 3!} &  n \equiv 1  \pmod{8}  \\[1.5mm]
				- 2 \: \frac{(n-2)\: n \: (n+2)}{4^3  \:3!} &  n \equiv 2  \pmod{8} \\[1.5mm]  
				-\frac{(n-3)(n+1)(n+5)}{4^3   \; 3!} &  n \equiv 3  \pmod{8}  \\[1.5mm]
				\quad 0 &  n \equiv 4  \pmod{8}  \\   
				+ 	\frac{(n-5)(n-1)(n+3)}{4^3 \: 3!} &  n \equiv 5  \pmod{8}  \\[1.5mm]
				+2 \: \frac{(n-2)\: n \: (n+2)}{4^3  \:3!} &  n \equiv 6  \pmod{8}  \\[1.5mm]  
				+ \frac{(n-3)(n+1)(n+5)}{4^3   \; 3!} &  n \equiv 7 \pmod{8}  
			\end{cases}, 
		\end{aligned} \\
	\end{equation}
	
	\begin{equation}
		\label{initial values - 3 } 
					\large
		\Psi_4(n) = 		
		\begin{cases}
			+ 2 \:  \frac{ (n-4) \: n^2 \:  (n+4)}{4^4  \; 4!}  &  n \equiv 0  \pmod{8} \\[1.5mm]   
			+  \frac{(n-5)(n-1)(n+3)(n+7)}{4^4  \; 4!} &  n \equiv 1  \pmod{8}  \\[1.5mm]
			\quad 	0 &  n \equiv 2  \pmod{8}  \\   
			-  \frac{(n-7)(n-3)(n+1)(n+5)}{4^4  \; 4!} &  n \equiv 3  \pmod{8}  \\[1.5mm]
			-  \: 2 \:  \frac{ (n-4) \: n^2 \:  (n+4)}{4^4  \; 4!}  &  n \equiv 4  \pmod{8}  \\[1.5mm]
			- \frac{(n-5)(n-1)(n+3)(n+7)}{4^4  \; 4!} &  n \equiv 5  \pmod{8}  \\
			\quad 	0 &  n \equiv 6  \pmod{8}  \\[1.5mm]  
			+  \frac{(n-7)(n-3)(n+1)(n+5)}{4^4  \; 4!} &  n \equiv 7 \pmod{8}  
		\end{cases}, 
	\end{equation}

	and

	\begin{equation}
		\label{initial values - 4 } 
					\large
		\Psi_5(n) = 		
		\begin{cases}
			\quad 0  &  n \equiv 0  \pmod{8}  \\   
			+ \frac{(n-9)(n-5)(n-1)(n+3)(n+7)}{4^5  \; 5!} &  n \equiv 1  \pmod{8}  \\[1.5mm]
			+ \: 2 \: \frac{(n-6)(n-2) \: n \: (n+2)(n+6)}{4^5  \; 5!} &  n \equiv 2  \pmod{8} \\[1.5mm]  
			+  \frac{(n-7)(n-3)(n+1)(n+5)(n+9)}{4^5  \; 5!} &  n \equiv 3  \pmod{8}  \\
			\quad 0  &  n \equiv 4  \pmod{8}  \\   
			-\frac{(n-9)(n-5)(n-1)(n+3)(n+7)}{4^5  \; 5!} &  n \equiv 5  \pmod{8}  \\[1.5mm]
			-  \: 2 \: \frac{(n-6)(n-2) \: n \: (n+2)(n+6)}{4^5  \; 5!} &  n \equiv 6  \pmod{8}  \\[1.5mm]  
			-  \frac{(n-7)(n-3)(n+1)(n+5)(n+9)}{4^5  \; 5!} &  n \equiv 7 \pmod{8}  
		\end{cases}, 
	\end{equation}
	and 
	\begin{equation}
					\large
		\label{initial values - 5 } 
		\Psi_6(n) = 		
		\begin{cases}
			- \: 2 \: \frac{(n-8)(n-4) \: n^2 \: (n+4)(n+8)}{4^6  \; 6!}   &  n \equiv 0  \pmod{8}  \\[1.5mm]  
			-  \frac{(n-9)(n-5)(n-1)(n+3)(n+7)(n+11)}{4^6  \; 6!}  &  n \equiv 1  \pmod{8}  \\
			\quad 0   &  n \equiv 2  \pmod{8}  \\   
			+  \frac{(n-11)(n-7)(n-3)(n+1)(n+5)(n+9)}{4^6  \; 6!}  &  n \equiv 3  \pmod{8}  \\[1.5mm]
			+  2 \: \frac{(n-8)(n-4) \: n^2 \: (n+4)(n+8)}{4^6  \; 6!}   &  n \equiv 4  \pmod{8}  \\[1.5mm]  
			+  \frac{(n-9)(n-5)(n-1)(n+3)(n+7)(n+11)}{4^6  \; 6!}   &  n \equiv 5  \pmod{8}  \\
			\quad 0 &  n \equiv 6  \pmod{8} \\[1.5mm]  
			- \frac{(n-11)(n-7)(n-3)(n+1)(n+5)(n+9)}{4^6  \; 6!}  &  n \equiv 7 \pmod{8},  
		\end{cases},
	\end{equation}
	and the following example should give us a better vision about that sequence: 
	\begin{equation}
		\label{initial values - 7 } 
		\begin{aligned}	
			\large
			\Psi_7(n) = 		
			\begin{cases}
				\quad 	0  &  n \equiv 0  \pmod{8}  \\   
				
				-	\frac{(n-13)(n-9)(n-5)(n-1)(n+3)(n+7)(n+11)}{4^7 \: 7!} &  n \equiv 1  \pmod{8}  \\[1.5mm]
				
				- 2 \: \frac{(n-10)(n-6)(n-2)\: n \: (n+2)(n+6)(n+10)}{4^7  \:7!} &  n \equiv 2  \pmod{8} \\[1.5mm]  
				
				-\frac{(n-11)(n-7)(n-3)(n+1)(n+5)(n+9)(n+13)}{4^7   \; 7!} &  n \equiv 3  \pmod{8}  \\[1.5mm]
				
				\quad 0 &  n \equiv 4  \pmod{8}  \\   
				
				+ 	\frac{(n-13)(n-9)(n-5)(n-1)(n+3)(n+7)(n+11)}{4^7 \: 7!} &  n \equiv 5  \pmod{8}  \\[1.5mm]
				
				+2 \: \frac{(n-10)(n-6)(n-2)\: n \: (n+2)(n+6)(n+10)}{4^7  \:7!} &  n \equiv 6  \pmod{8}  \\[1.5mm]  
				
				+ \frac{(n-11)(n-7)(n-3)(n+1)(n+5)(n+9)(n+13)}{4^7   \; 7!} &  n \equiv 7 \pmod{8}  
			\end{cases}. 
		\end{aligned} \\
	\end{equation}
	As we see, sometimes the sequence has a center factor that all the other factors in the numerators get around it. For $n \equiv 0,2,4,6 \pmod{8}$ is symmetric in this sense. However,  it gives a different story for $n \equiv 1,3,5,7 \pmod{8}$,  and a natural phenomena for these numbers arise up here and need a closer attention. 
	\subsection{Right tendency for $\Psi_k(n)$}
	For $n \equiv \pm 1 \pmod{8}$, and from the data above, and from the formulas of $\Psi_k(n)$, we can observe that the factors of the numerators for $k=1,2,3,4,5,6, \cdots$ always fill the right part first then go around the center factor to fill the left part and so on as explained in the following pattern and table
	\begin{equation}{\text{Right Tendency For } \Psi_k(n) \quad  }
		\label{Right Tendency}
		\begin{array}{c  c  c c  c c c c c } 
			k &   & &   &   \text{The Center} &  &  &    \\
			
			1  & &  & &(n-1 )&  &  &   \\
			
			2 &  &  &   &  (n-1 ) & (n+3 )  &    &   \\
			
			3  & &  &(n-5 ) & (n-1 ) & (n+3 ) &   & \\
			
			4   &  &  &(n-5 ) &   (n-1 ) & (n+3) & (n+7) &   \\  
			
			5      &  &   (n-9) &(n-5 ) &   (n-1 ) & (n+3) & (n+7) &  \\  
			
			6      &  &   (n-9) &(n-5 ) &   (n-1 ) & (n+3) & (n+7) & (n+11)   \\  
		\end{array}
	\end{equation}

	\subsection{Left tendency for $\Psi_k(n)$}
	However, for $n \equiv \pm 3 \pmod{8}$, and from the data above, and the formulas of $\Psi_k(n)$, we can also observe that the factors of the numerators for $k=1,2,3,4,5,6, \cdots $ always fill the left part first then go around the center factor to fill the right part and so on as explained in the following pattern and table
	\begin{equation}{\text{Left Tendency For } \Psi_k(n) \quad  }
		\label{Left Tendency}
		\begin{array}{c  c  c c  c c c c c } 
			k &   & &  &  \text{The Center} &  &   &      \\
			
			1  & &  & &(n+1 )&  &  &     \\
			
			2 &  &  &  (n-3 )  &  (n+1 ) &  &    &   \\
			
			3  & &  &(n-3 ) & (n+1 ) & (n+5 ) &   &  \\
			
			4   &  & (n-7) &(n- 3) &   (n+1 ) & (n+5) &  &   \\  
			
			5    &  & (n-7) &(n- 3) &   (n+1 ) & (n+5) & (n+9) &   \\   
			
			6   & (n-11)  &  (n-7)&   (n-3) &(n+1 ) &   (n+5 ) & (n+9) &   \\  
		\end{array}
	\end{equation}

	\subsection{The signs phenomena of $\Psi-$integers}
 For $ k \equiv 0,\: 1  ,\:2  ,\:3  ,\: 4   ,\: 5  ,\: 6  ,\: 7    \pmod{8}$, we get the following data
	\begin{equation}
		\label{Signes1}
		\begin{array}{c  c  c c  c c } \\
			k \pmod{8} &  (-1)^{\lfloor{\frac{k}{2}}\rfloor}   & (-1)^{\lfloor{\frac{k}{2}}\rfloor +1} & (-1)^{\lfloor{\frac{k}{2}}\rfloor + \delta(k)}  &   (-1)^{\lfloor{\frac{k}{2}}\rfloor + \delta(k-1)} \\
			0 &  +1 &  -1 & +1  &  -1    \\
			
			1  &  +1  &  -1  &  -1   &  +1        \\
			
			2  &  -1  &  +1  &   -1     &  +1     \\
			
			3  &  -1  &   +1     &  +1    &  -1     \\
			
			4   &  +1  &   -1     &  +1    &   -1    \\      
			
			5    &   +1    & -1    &    -1   &  +1     \\
			
			6       &    -1    &  +1   &   -1   &   +1    \\
			
			7    &  -1  &   +1      &  +1    &    -1   \\
		\end{array}
	\end{equation}
	
		Therefore, the signs of $\Psi-$sequence get periodically every 8 steps. Table \eqref{Signes2} shows that three plus are followed by zero then followed by three minus then followed by zero and so on. For $ n \equiv 0,\: 1  ,\:2  ,\:3  ,\: 4   ,\: 5  ,\: 6  ,\: 7    \pmod{8}$ and for $ k \equiv 0,\: 1  ,\:2  ,\:3  ,\: 4   ,\: 5  ,\: 6  ,\: 7    \pmod{8}$  the following data is useful \\
	\begin{equation}{\text{Signs of } \Psi_k(n)  }
		\label{Signes2}
		\begin{array}{c  c  c c  c c c c c c} 
			n \pmod{8} &  \Psi_0(n) & \Psi_1(n) & \Psi_2(n)  &   \Psi_3(n) & \Psi_4(n) & \Psi_5(n)  &   \Psi_6(n) & \Psi_7(n)   \\
			0 &  + & 0 & - & 0 & + & 0 & -  &  0  \\
			
			1  & + & + & - & - & + & + & - &  - \\
			
			2 & 0 & + & 0 &  - &  0  & +   & 0 & -   \\
			
			3  & - & + & + & - &  -  &  + & + & - \\
			
			4   & - & 0 & + &   0 & - & 0 & + &  0  \\  
			
			5    & - & - & + & +  & - & - & + &  +  \\
			
			6     & 0 & - &  0 & + & 0 & - & 0 &  +  \\
			
			7    & + & - & - &  + & + & - & - &  + \\
		\end{array}
	\end{equation}

	\subsection{General property for the $\Psi-$integers}
	Whether $ n \equiv 0,\: 1  ,\:2  ,\:3  ,\: 4   ,\: 5  ,\: 6  ,\: 7    \pmod{8},  $  it is rather surprising that the quantity 
	\[    \Psi_{k}(n) \div       \Psi_{k-2}(n)        \]
	always gives 
	\[  - \: \:\frac{\: \: n^2 \: - \:  (2k-4)^2 \:}{16 \: k \: (k-1)} \quad   \text{or} \quad 
	- \: \frac{\: \: (n     \pm 1) ^2 \: - \:  (2k-2)^2 \:}{16 \: k \: (k-1)} \quad   \text{or} \quad  0 \div 0.    \]
	Therefore, computing these ratios, at the eight levels $ n \equiv 0,\: 1  ,\:2  ,\:3  ,\: 4   ,\: 5  ,\: 6  ,\: 7    \pmod{8},  $ immediately proves the following desirable theorem that gives a fraction with a difference of two squares divided by a multiplication of the factors $16$,  $k$, and $k-1$.
	
	\begin{theorem}{(Generating The $\Psi-$integers From The Previous Term)}
		\label{Generating psi}  \\
		If $ \Psi_{k}(n)$ not identically zero, then
		\[   \frac{\Psi_{k}(n)}{\Psi_{k-2}(n)} = - \: \frac{\: \:		\big[n +(-1)^{\lfloor{\frac{n}{2}}\rfloor + k  } \delta(n) \big]^2 \: - \: \big[2k-2 - 2 \delta(n-1) \big]   ^2 \:}{16 \:\: k \: (k-1)},
		\]
		and we can choose either of the following initial values to generate $\Psi_k(n)$ from the starting term $\Psi_0(n)$ or the last term $\Psi_{\lfloor{\frac{n}{2}}\rfloor }(n)$ :
		\begin{equation*}
			\label{starting values } 
			\begin{aligned}	
				\Psi_0(n) = 		
				\begin{cases}
					+2  &  n \equiv \:\: 0  \pmod{8}  \\   
					+	 1 &  n \equiv \pm 1  \pmod{8}  \\
					\: \: 0  &  n \equiv \pm 2  \pmod{8}  \\   
					-1&  n \equiv \pm 3  \pmod{8}  \\
					-2  &  n \equiv \pm 4  \pmod{8}  \\   
				\end{cases} 
				\qquad , \qquad 
				\Psi_{\lfloor{\frac{n}{2}}\rfloor }(n) = 1.
			\end{aligned} \\
		\end{equation*}
	\end{theorem}

	\section{The emergence of $ \phi-$sequence}
	Another natural sequence that emerges naturally from $\Psi_k(n) $ is the integer sequence $\phi_k(n)$ which is defined as follows

	\begin{definition} 
		\[        \phi_k(n) := 4^k \: \Psi_k(n)    \]
	\end{definition}

	\subsection{Recurrence relation of order $4$ to generate $\phi-$sequence}
	
	From \eqref{lemma one}, we get the following recurrence relation 
	
	\begin{lemma}
		\label{lemma 1B}  
		For each natural number $n$, $ 0 \le k \le  \lfloor{\frac{n}{2}}\rfloor$,  the integers  $\phi_k(n)$  satisfy the following property
		\begin{equation}
			\label{Equation of result-4I} 
			\phi_k(n) = 4 \: \phi_{k-1}(n-2)  -    \phi_k(n-4)       	
		\end{equation}
	\end{lemma}
	From \eqref{initial values - 1}, we easily get the following initial values
	\begin{lemma}
		\label{2B}
		For each natural number $n$
		
		\begin{equation}
			\label{initial values - 1B} 
			\begin{aligned}	
				\phi_0(n) = 		
				\begin{cases}
					+2  &  n \equiv 0  \pmod{8}  \\   
					+	 1 &  n \equiv \pm 1  \pmod{8}  \\
					\: \: 0  &  n \equiv \pm 2  \pmod{8}  \\   
					-1&  n \equiv \pm 3  \pmod{8}  \\
					-2  &  n \equiv \pm 4  \pmod{8}  \\   
				\end{cases} 
				\qquad , \qquad 
				\phi_1(n) = 		
				\begin{cases}
					\quad 	0  &  n \equiv 0  \pmod{8}  \\   
					+	(n-1) &  n \equiv 1  \pmod{8}  \\
					+ 2 \: n &  n \equiv 2  \pmod{8}  \\   
					+ (n+1) &  n \equiv 3  \pmod{8}  \\
					\quad 0 &  n \equiv 4  \pmod{8}  \\   
					- (n-1) &  n \equiv 5  \pmod{8}  \\
					-2 \:  n  &  n \equiv 6  \pmod{8}  \\   
					- (n+1)  &  n \equiv 7 \pmod{8}  
				\end{cases}, 
			\end{aligned} \\
		\end{equation}
	\end{lemma}

	\subsection{Explicit formulas For $\phi-$sequence}

	Now, from Theorem \eqref{Theorem of the 4 levels}, we get the following explicit formulas for the integer sequence $\phi_k(n)$.
	
	\begin{lemma}
		\label{values of phi}  
		For any non negative integers $n,k$, the  sequences $\phi_k(n)$ satisfy the following statements \\
		
			\underline{$n \equiv 0 \pmod{8}$} 
		\begin{equation*}
	 		\phi_k(n) = \begin{cases}
					0  & \mbox{for $k$ odd } \\   
					2 \: (-1)^{\lfloor{\frac{k}{2}}\rfloor}
					\: \frac {\prod\limits_{\lambda = 0}^{\lfloor{\frac{k}{2}}\rfloor -1} n^2 - (4 \lambda)^2}	{ k!} & \mbox{for $k$ even } \end{cases}   \\
				  	\end{equation*}
			  			  	
	\underline{$n \equiv 1 \pmod{8}$} 
				\begin{equation*}				
				\phi_k(n) = (-1)^{\lfloor{\frac{k}{2}}\rfloor} (n+1-2k)^{\delta(k)}  \:
				\frac {\prod\limits_{\lambda = 1}^{\lfloor{\frac{k}{2}}\rfloor} (n+1)^2 - (4 \lambda -2)^2}	{k!}  \\
				  	\end{equation*}
			  	
	\underline{$n \equiv 2 \pmod{8}$} 	  	
				\begin{equation*}
				\phi_k(n) =
				\begin{cases}
					0  & \mbox{for $k$ even } \\   
					2 \: (-1)^{{\lfloor{\frac{k}{2}}\rfloor}} \:
					n \:	\frac {\prod\limits_{\lambda = 1}^{\lfloor{\frac{k}{2}}\rfloor} n^2 - (4 \lambda -2)^2}	{ k!} & \mbox{for $k$ odd } \end{cases}  
				\end{equation*}
			
	\underline{$n \equiv 3 \pmod{8}$} 	
				\begin{equation*}
				\phi_k(n) = (-1)^{\lfloor{\frac{k}{2}}\rfloor + \delta(k-1)} (n+1) (n+1-2k)^{\delta(k-1)}  \:
				\frac {\prod\limits_{\lambda = 1}^{\lfloor{\frac{k}{2}}\rfloor - \delta(k-1)} (n+1)^2 - (4 \lambda)^2}	{ k!}  \\ 
					\end{equation*}
			
	\underline{$n \equiv 4 \pmod{8}$} 	
				\begin{equation*}
				\phi_k(n) =
				\begin{cases}
					0  & \mbox{for $k$ odd } \\   
					2 \: (-1)^{\lfloor{\frac{k}{2}}\rfloor +1}  \:
					\frac {\prod\limits_{\lambda = 0}^{\lfloor{\frac{k}{2}}\rfloor -1} n^2 - (4 \lambda)^2}	{ k!} & \mbox{for $k$ even } \end{cases}   \\ 
					\end{equation*}
			
	\underline{$n \equiv 5 \pmod{8}$} 	
				\begin{equation*}	 
				\phi_k(n) = (-1)^{\lfloor{\frac{k}{2}}\rfloor +1}  \: (n+1-2k)^{\delta(k)} \:
				\frac {\prod\limits_{\lambda = 1}^{\lfloor{\frac{k}{2}}\rfloor }(n+1)^2 - (4 \lambda -2)^2}	{ k!}  \\
					\end{equation*}
			
	\underline{$n \equiv 6 \pmod{8}$} 	
				\begin{equation*} 	
				 \phi_k(n) =
				\begin{cases}
					0  & \mbox{for $k$ even } \\   
					2 \: (-1)^{\lfloor{\frac{k}{2}}\rfloor +1}
					\:  n \: \frac {\prod\limits_{\lambda = 1}^{\lfloor{\frac{k}{2}}\rfloor } n^2 - (4 \lambda -2)^2}	{ k!} & \mbox{for $k$ odd } \end{cases}   
					\end{equation*}
				
	\underline{$n \equiv 7 \pmod{8}$} 	
				\begin{equation*}
		   \phi_k(n) = (-1)^{\lfloor{\frac{k}{2}}\rfloor + \delta(k)} (n+1) (n+1-2k)^{\delta(k-1)}  \:
				\frac {\prod\limits_{\lambda = 1}^{\lfloor{\frac{k}{2}}\rfloor - \delta(k-1)} (n+1)^2 - (4 \lambda)^2}	{ k!}  
		\end{equation*}
	
	\end{lemma}

	\subsection{Nonlinear recurrence relation  to generate $\phi-$sequence}
	
	To study the arithmetic properties of $\phi-$Sequence, we need to generate $\phi_k(n)$ from the previous one, $\phi_{k-2}(n)$, or generate $\phi_{k}(n)$ from the next one, $\phi_{k+2}(n)$. As 
	\[     \frac{\Psi_{k}(n)}{\: \: \Psi_{k-2}(n)}   = \frac{\: \phi_{k}(n)}{ \: 4^2  \: \phi_{k-2}(n)},                          \]
	from Theorem \eqref{Generating psi}, we get the following desirable theorem.
	\begin{theorem}{(Generating The $\phi-$integers From The Previous Term)}
		\label{Generating phi}  \\
		If $ \phi_{k}(n)$ not identically zero, $n$ even, then
		\[   \frac{\phi_{k}(n)}{\phi_{k-2}(n)} = - \: \frac{\: \:	n^2  \: - \: (2k-4)^2 \:}{\: k \: (k-1)},
		\]
		and depending on the parity of $k$ we can choose either of the following initial values to generate $\phi_k(n)$ from the starting terms $\phi_0(n)$ or $\phi_1(n)$:
		\begin{equation}
			\label{initial values - C} 
			\begin{aligned}	
				\phi_0(n) = 		
				\begin{cases}
					+2  &  n \equiv 0  \pmod{8}  \\   
					-2  &  n \equiv  4  \pmod{8}  \\   
				\end{cases} 
				\qquad , \qquad 
				\phi_1(n) = 		
				\begin{cases}
					+ 2 \: n &  n \equiv 2  \pmod{8}  \\     
					-2 \:  n  &  n \equiv 6  \pmod{8}  \\   
				\end{cases}, 
			\end{aligned} \\
		\end{equation}
	\end{theorem}

	\section{An overview of applications towards Mersenne primes}

	From Lemmas \eqref{lemma 1B}, \eqref{2B},\eqref{values of phi}, we ready to prove the following theorem.
	
	\subsection{Primality tests for Mersenne numbers}
	
	\begin{theorem}{(Lucas-Lehmer-Moustafa)}{(Version 1)}
		\label{Theorem of result-33}\\
		Given prime $p \geq 5$. $2^p-1$ is prime  \textbf{ if and only if}	
		\begin{equation}
			\label{Equation of result-33} 
			\Large	2^p-1 \quad  \vert \quad \sum_{\substack{k=0,\\ k \:even}}^{\lfloor{\frac{n}{2}}\rfloor}  \phi_k (n)
		\end{equation}
		where $n:=2^{p-1}$,  $\phi_k (n)$ are defined by the double index recurrence relation \[\phi_k (m)= 4 \: \phi_{k-1}(m-2) - \phi_k (m-4) \]
		and the initial boundary values satisfy 
		\begin{equation}
			\label{initial values - 1BB} 
			\begin{aligned}	
				\phi_0(m) = 		
				\begin{cases}
					+2  &  m \equiv 0  \pmod{8}  \\   
					\: \: 0  &  m \equiv \pm 2  \pmod{8}  \\  
					-2  &  m \equiv 4  \pmod{8}  \\   
				\end{cases} 
				\qquad , \qquad 
				\phi_1(m) = 		
				\begin{cases}
					+ 2 \: m &  m \equiv 2  \pmod{8}  \\   
					\quad 0 &  m \equiv 0, \:4  \pmod{8}  \\   
					-2 \:  m  &  m \equiv 6  \pmod{8}  \\    
				\end{cases}, 
			\end{aligned} \\
		\end{equation}
		
	\end{theorem}
	
	\begin{proof}
		Given prime $p \geq 5$, let $n:=2^{p-1}$. From Lucas-Lehmer-Test, see \cite{Jean} and \cite{Elina}, we have \\
		\begin{equation*}
			\begin{aligned}
				2^p -1 \quad \text{is prime} \quad  \iff 2^p -1  \quad | \quad (1+\sqrt{3})^n + (1-\sqrt{3})^n.	
			\end{aligned}
		\end{equation*}
		Hence, as $n \equiv 0 \pmod{8}$, replace $x= 1+\sqrt{3}, \quad  y= 1-\sqrt{3}$ in Theorem \eqref{Theorem of the 4 levels}, we get the following equivalent statements:
		\begin{equation*}
			\begin{aligned}
				2^p -1 \quad \text{is prime} \quad  &\iff 2^p -1  \quad | \quad \sum_{\substack{k=0,\\ k \:even}}^{\lfloor{\frac{n}{2}}\rfloor}  \Psi_k(n) \:
				(-2)^{\lfloor{\frac{n}{2}}\rfloor -k} (8)^{k} \\ 
				&\iff 2^p -1  \quad | \quad 2^{\lfloor{\frac{n}{2}}\rfloor} \sum_{\substack{k=0,\\ k \:even}}^{\lfloor{\frac{n}{2}}\rfloor}  \Psi_k(n) \quad 4^k\\ 	
				&\iff 2^p -1  \quad | \quad  \sum_{\substack{k=0,\\ k \:even}}^{\lfloor{\frac{n}{2}}\rfloor}  \phi_k(n). \\ 
			\end{aligned}
		\end{equation*}
		
		Lemmas \eqref{lemma 1B}, \eqref{2B} already proved the rest of Theorem \eqref{Theorem of result-33}. 
	\end{proof}
	
	From Lemmas \eqref{lemma 1B}, \eqref{2B}, we should observe that the recurrence relation of	$\phi_k(n)$ is always even integer for  $n \equiv 0 \pmod{8}$. Hence from Lemma \eqref{values of phi}, we get the following theorem.

	\begin{theorem}{(Lucas-Lehmer-Moustafa)}{(Version 2) }
		\label{Theorem of result-333}\\
		Given prime $p \geq 5$,  $n:=2^{p-1}$. The number $2^p-1$ is prime  \textbf{ if and only if} 	
		\begin{equation}
			\label{Equation of result-333} 
			2\:n -1  \quad  \vert \quad \sum_{\substack{k=0,\\ k \:even}}^{\lfloor{\frac{n}{2}}\rfloor}   \quad  (-1)^{\lfloor{\frac{k}{2}}\rfloor}	\: \frac {\prod\limits_{\lambda = 0}^{\lfloor{\frac{k}{2}}\rfloor -1} n^2 - (4 \lambda)^2}	{ k!}. 
		\end{equation}
	\end{theorem}

	\subsection{Generating $\phi-$integers}
	When we generate the $\phi-$integers needed for Mersenne numbers, we should notice that for $n:=2^p, p\geq 5,$ we have  \[n \equiv 0 \pmod{8}.\]
	Hence we get the following theorem. 
	
	\begin{theorem}{(Generating $\phi-$integers without changing $n$)}
		\label{Generating phi for Mersenne}  \\
		For $n:=2^{p-1}, p\geq 5,$ then

		\begin{equation}
			\label{rec phi}	
			\frac{\phi_{k}(n)}{\phi_{k-2}(n)} = - \: \frac{\: \:	n^2  \: - \: (2k-4)^2 \:}{\: k \: (k-1)},
		\end{equation}
				and we can choose either of the following initial values to generate $\phi_k(n)$ from the starting term $\phi_0(n)$ or the last term $\phi_{\lfloor{\frac{n}{2}}\rfloor }(n)$ :
		\begin{equation*}
			\label{special starting values of phi} 
			\begin{aligned}	
				\phi_0(n) = +2 		
				\qquad , \qquad 
				\phi_{\lfloor{\frac{n}{2}}\rfloor }(n) = 4^{\lfloor{\frac{n}{2}}\rfloor}.
			\end{aligned} \\
		\end{equation*}
		
	\end{theorem}

	\begin{theorem}{(Lucas-Lehmer-Moustafa)}{(Version 3)}
		\label{Theorem of result-4}\\
		Given prime $p \geq 5, \: n:=2^{p-1}$. The number $2^p-1$ is prime  \textbf{ if and only if} 	
		\begin{equation}
			\label{Equation of result-44} 
				2\: n -1 \quad  \vert \quad \sum_{\substack{k=0,\\ k \:even}}^{\lfloor{\frac{n}{2}}\rfloor}  \phi_k (n)
		\end{equation}
		where $\phi_k (n)$ are defined and generated by the double index recurrence relation
		\[   \frac{\phi_{k}(n)}{\phi_{k-2}(n)} = - \: \frac{\: \:	n^2  \: - \: (2k-4)^2 \:}{\: k \: (k-1)},
		\]
		and we can choose either of the following initial values to generate $\phi_k(n)$ from the starting term $\phi_0(n)$ or the last term $\phi_{\lfloor{\frac{n}{2}}\rfloor }(n)$ :
		\begin{equation*}
			\label{starting values of phi for mersenne} 
			\begin{aligned}	
				\phi_0(n) = +2 		
				\qquad , \qquad 
				\phi_{\lfloor{\frac{n}{2}}\rfloor }(n) = 4^{\lfloor{\frac{n}{2}}\rfloor}.
			\end{aligned} \\
		\end{equation*}
		
	\end{theorem}

	\subsection{Criteria for compositeness of Mersenne numbers }
	
	The following theorem is an immediate consequence of Theorem \eqref{Theorem of result-333}
	
	\begin{theorem}{(Criteria for compositeness of Mersenne numbers ) }

		\label{Criteria for Compositeness of Mersenne Numbers}
		Given prime $p \geq 5$. The number  $2n-1=2^p -1$ is Mersenne composite number if
		\begin{equation}
			\label{Equation of result-3333} 
			\large	2 \: n \: - \: 1 \quad  \nmid \quad \sum_{\substack{k=0,\\ k \:even}}^{\lfloor{\frac{n}{2}}\rfloor}   \quad  	\: \frac {\prod\limits_{\lambda = 0}^{\lfloor{\frac{k}{2}}\rfloor -1}  [(4 \lambda)^2 \: - \: 4^{-1}] \quad}	{ k!}. 
		\end{equation}
	\end{theorem} 	 
	
	\section{Some formulas and possible scenario}
	Now, consider $p$ prime greater than 3, and  $n:=2^{p-1}$. The previous sections give various explicit formulas and techniques to compute and generate all of the terms $\phi_k(n)$ needed for checking the primality of the Mersenne number $2^p-1$. 
	
	\subsection{Formulas}
	Now we compute the summation 
	\begin{equation}
		\label{summation of phi} 
		\Large \sum_{\substack{k=0,\\ k \:even}}^{\lfloor{\frac{n}{2}}\rfloor}  \phi_k (n)
	\end{equation}
	for the first few terms from the starting and from the end. From Theorem \eqref{Generating phi for Mersenne}, we get the following explicit terms 
	\[   \phi_0(n) = \: +\:2,\]
	Now, compute $\phi_2(n)$ as following
	\[   \phi_2(n) = - \: \frac{\: \:	n^2  \: - \: (4-4)^2 \:}{\:2 \: (2-1)} \: \phi_0(n)  =  - \: n^2.\]
	Proceeding this way, we get 
	\[ \phi_4(n) = - \: \frac{\: \:	n^2  \: - \: (8-4)^2 \:}{\:4 \: (4-1)} \: \phi_2(n)  =  + \: \frac{n^2(n^2 - 4^2)}{12}.\]
	Similarly 
	\[ \phi_6(n) =  - \: \frac{n^2(n^2 - 4^2)(n^2 - 8^2)}{360},  \]
	
	\[ \phi_8(n) = +  \: \frac{n^2(n^2 - 4^2)(n^2 - 8^2)(n^2 - 12^2)}{20160}, \quad etc               \]

	Now, compute $\phi_k(n)$ from the end, and from Theorem \eqref{Generating phi for Mersenne}, we get	   
	\[   \phi_{k-2}(n) = - \: \frac{\: k \: (k-1)}{\: \:	n^2  \: - \: (2k-4)^2 \:} \phi_{k}(n). \]
	Initially
	\[    \phi_{\lfloor{\frac{n}{2}}\rfloor }(n) = 2^n,     \]
	Hence
	\[    \phi_{\lfloor{\frac{n}{2}}\rfloor -2}(n) = - \: n \: 2^{n-5},     \]
	and 
	\[    \phi_{\lfloor{\frac{n}{2}}\rfloor -4}(n) = + \: n \: (n-6) \: 2^{n- 11}.     \]
	Similarly  
	\[    \phi_{\lfloor{\frac{n}{2}}\rfloor -6}(n) = - \: \frac{n(n-8) (n-10)}{3} \: 2^{n- 16} ,    \]
	
	\[ \phi_{\lfloor{\frac{n}{2}}\rfloor -8}(n) = + \: \frac{n(n-10) (n-12)(n-14)}{3} \: 2^{n- 23}\quad  \quad etc.  \]
	
	The author feels that we need a clever way to evaluate the sum \eqref{summation of phi}. We may like to add the terms in a way reflects some elegant arithmetic. Remember that we do not need to compute the sum \eqref{summation of phi} exactly; but we just need to find the sum modulo $2n-1$. According to the following theorem, and working modulo $2n-1$, the last term always gives the value of the first term.
	\begin{theorem}
		For $p \geq 5$ prime, and $n:=2^{p-1}$,
		\begin{equation}
			\label{last value of phi} 
			\phi_{\lfloor{\frac{n}{2}}\rfloor }(n)  \equiv 2  \pmod{2n-1}. 
		\end{equation}  
	\end{theorem}
	\begin{proof}
		We should observe that if $p \geq 5$ prime, then $n=2^{p-1} = 1 + \zeta \: p$, for some positive integer $\zeta.$ Then
		\[\phi_{\lfloor{\frac{n}{2}}\rfloor }(n) =  2^n \: = \: 2^{1 + \zeta \: p} = 2^1 \: (2^p)^\zeta  \equiv 2  \pmod{2n-1}.  \] 
	\end{proof}
	Hence this encourages one to compute the partial sums of
	\begin{equation}
		\label{summation of phi mod} 
		\sum_{\substack{k=0,\\ k \:even}}^{\lfloor{\frac{n}{2}}\rfloor}  \phi_k (n) \pmod {2n-1},
	\end{equation}
	in the following order. 
	\subsection{The $5$ scenario}
	For example, take $p=5$, then $n=2^4$. Hence
	\begin{equation}
		\label{example for sum of phi} 
		\begin{aligned}	
			\phi_0(31)&\equiv\: +2 \pmod {31}, &  \phi_2(31)&\equiv  \; - \; 2^3  \pmod {31}, \\  \phi_4(31)&\equiv+2^2 +1  \pmod {31},   &  \phi_6(31)&\equiv -1  \pmod {31},  \\
		\end{aligned} 
	\end{equation}  
	and $ \phi_8(31) \equiv \: + 2 \pmod {31}$. Hence we get the partial sums 
	\begin{equation}
		\label{scenario 1} 
		\begin{aligned}	
			\phi_{8}(31)&\equiv + 2,\\
			\phi_{8}(31) + \phi_0(31)&\equiv + 2^2, \\
			\phi_{8}(31) + \phi_0(31)  +  \phi_2(31) &\equiv - 2^2, \\
			\phi_{8}(31) + \phi_0(31)  +  \phi_2(31) + \phi_4(31)&\equiv  + 1, \\
			\phi_{8}(31) + \phi_0(31)  +  \phi_2(31) + \phi_4(31) +\phi_6(31) &\equiv  \: 0. 
		\end{aligned} 
	\end{equation}
	As we ended up with zero, this shows that $2^5 -1 = 31$ is Mersenne prime. This particular example should motivate us for more theoretical investigations for other similar scenarios for this particular pattern that may occur in general for other cases for the partial sums
	\begin{equation}
		\label{sums in order} 
		\begin{aligned}	
			&\phi_{\lfloor{\frac{n}{2}}\rfloor }(n), \\ 
			&\phi_{\lfloor{\frac{n}{2}}\rfloor }(n) + \phi_0(n), \\ 
			&\phi_{\lfloor{\frac{n}{2}}\rfloor }(n) + \phi_0(n) + \phi_2(n), \\ 
			& \phi_{\lfloor{\frac{n}{2}}\rfloor }(n) + \phi_0(n) + \phi_2(n) + \phi_4(n),\\
			&\cdots 
		\end{aligned} 
	\end{equation}

	\section{Factoring factorial in terms of difference of squares}
	
	Choosing $k=\lfloor{\frac{n}{2}}\rfloor$ in the Eight Levels Theorem \eqref{Theorem of the 4 levels}, and noting $\Psi_{\lfloor{\frac{n}{2}}\rfloor }(n) = 1$, we surprisingly get the following eight combinatorial identities which reflect some unexpected facts about the nature of numbers.  
	\begin{equation*}
		\begin{array}{l  l l}	
			\underline{n \equiv 0 \pmod{8}}  &  \quad & 4^{\lfloor{\frac{n}{2}}\rfloor} \: (\lfloor{\frac{n}{2}}\rfloor)! = 	2 \:
			\:\prod\limits_{\lambda = 0}^{\lfloor{\frac{n}{4}}\rfloor -1} n^2 \: - \: (4 \lambda)^2 		\\
			
			\underline{n \equiv 1 \pmod{8}}  &  \quad&  4^{\lfloor{\frac{n}{2}}\rfloor} \: (\lfloor{\frac{n}{2}}\rfloor)! = 	\:
			\:\prod\limits_{\lambda = 1}^{\lfloor{\frac{n-1}{4}}\rfloor } (n+1)^2 \: - \: (4 \lambda -2)^2 	  
			\\
			\underline{n \equiv 2 \pmod{8}}  &  \quad& 4^{\lfloor{\frac{n}{2}}\rfloor} \: (\lfloor{\frac{n}{2}}\rfloor)! = 2	\: \: n 
			\:\prod\limits_{\lambda = 1}^{\lfloor{\frac{n-2}{4}}\rfloor } n^2 \: - \: (4 \lambda -2)^2 	  		
			\\  
			\underline{n \equiv 3 \pmod{8}}  &  \quad& 4^{\lfloor{\frac{n}{2}}\rfloor} \: (\lfloor{\frac{n}{2}}\rfloor)! =  \: (n+1) 
			\:\prod\limits_{\lambda = 1}^{\lfloor{\frac{n-3}{4}}\rfloor } (n+1)^2 \: - \: (4 \lambda)^2 	  
			\\
			\underline{n \equiv 4 \pmod{8}}  & \quad& 4^{\lfloor{\frac{n}{2}}\rfloor} \: (\lfloor{\frac{n}{2}}\rfloor)! =  \: 2 
			\:\prod\limits_{\lambda = 0}^{\lfloor{\frac{n-4}{4}}\rfloor } n^2 \: - \: (4 \lambda)^2 	  
			\\ 		
			\underline{n \equiv 5 \pmod{8}}  & \quad& 4^{\lfloor{\frac{n}{2}}\rfloor} \: (\lfloor{\frac{n}{2}}\rfloor)! =  
			\:\prod\limits_{\lambda = 1}^{\lfloor{\frac{n-1}{4}}\rfloor } (n+1)^2 \: - \: (4 \lambda -2)^2 	  
			\\
			\underline{n \equiv 6 \pmod{8}}  &  \quad& 4^{\lfloor{\frac{n}{2}}\rfloor} \: (\lfloor{\frac{n}{2}}\rfloor)! = 2	\: \: n 
			\:\prod\limits_{\lambda = 1}^{\lfloor{\frac{n-2}{4}}\rfloor } n^2 \: - \: (4 \lambda -2)^2 	  		
			\\  
			\underline{n \equiv 7 \pmod{8}}  &   \quad& 4^{\lfloor{\frac{n}{2}}\rfloor} \: (\lfloor{\frac{n}{2}}\rfloor)! =  \: (n+1) 
			\:\prod\limits_{\lambda = 1}^{\lfloor{\frac{n-3}{4}}\rfloor } (n+1)^2 \: - \: (4 \lambda)^2 	  
			\\
		\end{array} 
	\end{equation*}
	Writing only the different identities, which are 4, we get the following theorem which gives formulas and links to factorize any factorial in terms of a product of difference of squares.

	\begin{theorem}{(New combinatorial identities)}
		\label{New Combinatorial Identities}
		For any natural number $n$, the following combinatorial identities are correct
		\begin{equation*}
			\begin{array}{l  l l}	
				\underline{n \equiv 0 \pmod{4}}  &  \quad & 4^{\lfloor{\frac{n}{2}}\rfloor} \: (\lfloor{\frac{n}{2}}\rfloor)! = 	2 \:
				\:\prod\limits_{\lambda = 0}^{\lfloor{\frac{n-4}{4}}\rfloor } n^2 \: - \: (4 \lambda)^2 	\\[1.5mm]
				
				\underline{n \equiv 1 \pmod{4}}  &  \quad&  4^{\lfloor{\frac{n}{2}}\rfloor} \: (\lfloor{\frac{n}{2}}\rfloor)! = 	\:
				\:\prod\limits_{\lambda = 1}^{\lfloor{\frac{n-1}{4}}\rfloor } (n+1)^2 \: - \: (4 \lambda -2)^2 	  
				\\[1.5mm]	
				\underline{n \equiv 2 \pmod{4}}  &  \quad& 4^{\lfloor{\frac{n}{2}}\rfloor} \: (\lfloor{\frac{n}{2}}\rfloor)! = 2	\: \: n 
				\:\prod\limits_{\lambda = 1}^{\lfloor{\frac{n-2}{4}}\rfloor } n^2 \: - \: (4 \lambda -2)^2 	  		
				\\[1.5mm]
				\underline{n \equiv 3 \pmod{4}}  &  \quad& 4^{\lfloor{\frac{n}{2}}\rfloor} \: (\lfloor{\frac{n}{2}}\rfloor)! =  \: (n+1) 
				\:\prod\limits_{\lambda = 1}^{\lfloor{\frac{n-3}{4}}\rfloor } (n+1)^2 \: - \: (4 \lambda)^2 	  
				\\
			\end{array} \\
		\end{equation*}
	\end{theorem}

	\subsection*{Supplementary information}
	
	Data sharing not applicable to this article as no datasets were generated or analyzed during the current study.

	\subsection*{Conflict of interest}
	The author declares that he has no conflict of interest.

	\section*{Acknowledgments}
	
	I would like to deeply thank University of Bahrain for their support.

	\medskip
\end{document}